\documentclass{biometrika}

\usepackage{amsmath,amssymb,mathtools,bm}
\usepackage{microtype}
\usepackage{booktabs}
\usepackage[dvipsnames]{xcolor}
\usepackage[colorlinks=true,citecolor=MidnightBlue,linkcolor=MidnightBlue,urlcolor=MidnightBlue]{hyperref}
\usepackage[nameinlink,noabbrev]{cleveref}
\hypersetup{
  pdftitle={Sharp proper estimation of fixed-component Gaussian location mixtures in polynomial time},
  pdfauthor={Hengzhi He and Guang Cheng},
  pdfsubject={Proper density estimation for high-dimensional Gaussian mixtures},
  pdfkeywords={Gaussian mixture, Hellinger distance, Hermite tensor, moment method, proper learning}
}

\newcommand{\R}{\mathbb R}
\newcommand{\E}{\mathbb E}
\newcommand{\Pp}{\mathbb P}
\newcommand{\op}{\mathrm{op}}
\newcommand{\tr}{\mathrm{tr}}

\newcommand{\Sym}{\operatorname{Sym}}
\newcommand{\Span}{\operatorname{span}}
\newcommand{\cH}{\mathcal H}
\newcommand{\cG}{\mathcal G}

\newcommand{\cW}{\mathcal W}
\newcommand{\cC}{\mathcal C}
\newcommand{\eps}{\varepsilon}
\newcommand{\ip}[2]{\langle #1,#2\rangle}
\newcommand{\norm}[1]{\lVert #1\rVert}
\newcommand{\fnorm}[1]{\lVert #1\rVert_{\mathrm F}}

\newcommand{\contr}{\mathbin{\lrcorner}}

\numberwithin{equation}{section}

\makeatletter
\def\ps@plain{\ps@empty}
\let\ps@biom\ps@empty
\let\ps@biometrika\ps@empty
\makeatother
\begin{document}

\jname{}
\jyear{}
\copyrightinfo{}

\title{Sharp proper estimation of fixed-component Gaussian location mixtures in polynomial time}
\author{Hengzhi He and Guang Cheng}
\affil{Department of Statistics and Data Science, University of California, Los Angeles\
\email{hengzhihe@g.ucla.edu; guangcheng@stat.ucla.edu}}

\markboth{H. He and G. Cheng}{Sharp proper Gaussian mixture estimation}
\maketitle

\begin{abstract}
We consider a mixture of at most $k$ unit-covariance Gaussians in $\R^d$ whose
means belong to a fixed-radius ball, with no separation or minimum-weight
condition.  \citet{DWYZ23} proved that the minimax Hellinger risk is
of order $\sqrt{d/n}\wedge1$ and constructed a proper polynomial-time estimator
with the slower general bound $(d/n)^{1/4}$; obtaining the sharp rate in
polynomial time for fixed $k\ge3$ was left open.  We resolve this question.
The key device is a moment-fiber range finder.  A second-moment subspace controls
the energy missed by projection.  We then estimate finitely many one-free-index
Hermite contractions.  These vector-valued contractions recover every tensor
component containing exactly one missed direction at the sharp $\sqrt{d/n}$
scale.  Every term that remains contains at least two missed factors and is
therefore controlled by residual second-moment energy.  The resulting subspace
has dimension depending only on $k$.  Exhaustive moment fitting in this
constant-dimensional space produces a proper mixture and, together with the
dimension-free moment characterization of Gaussian mixtures, achieves the
optimal Hellinger rate in polynomial arithmetic time for every fixed $k$.
\end{abstract}


\section{Model and main result}

Let
\begin{equation}\label{eq:model}
  X=U+Z,\qquad Z\sim N(0,I_d),\qquad
  U\sim\Gamma=\sum_{j=1}^{s}w_j\delta_{\mu_j},
\end{equation}
where $s\le k$, $w_j\ge0$, $\sum_jw_j=1$, and $\norm{\mu_j}\le R$.
Write $P_\Gamma=\Gamma*N(0,I_d)$ and use the convention
\[
 H^2(P,Q)=\int(\sqrt p-\sqrt q)^2,
\]
so $0\le H(P,Q)\le\sqrt2$.  Let $B_R^d$ denote the closed Euclidean ball
of radius $R$.

\newpage
\begin{theorem}[Sharp proper polynomial-time estimation]\label{thm:main}
Fix $k\in\mathbb N$ and $R<\infty$.  There is a sample-splitting estimator
$\widehat\Gamma^{\mathrm E}$, supported on at most $k$ points of $B_R^d$, such
that, uniformly over all mixing distributions $\Gamma$ as in
\eqref{eq:model},
\begin{equation}\label{eq:main-expectation}
 \E_\Gamma H(P_{\widehat\Gamma^{\mathrm E}},P_\Gamma)
 \le C_{k,R}\min\left\{1,\sqrt{\frac dn}\right\}.
\end{equation}
Moreover, for every $\delta\in(0,1/3)$ there is a confidence-dependent
estimator $\widehat\Gamma^{\delta}$, also supported on at most $k$ points of
$B_R^d$, such that
\begin{equation}\label{eq:main-tail}
 \Pp_\Gamma\left\{
 H(P_{\widehat\Gamma^{\delta}},P_\Gamma)>
 C_{k,R}\min\left\{1,\sqrt{\frac{d+\log(1/\delta)}{n}}\right\}
 \right\}\le\delta.
\end{equation}
For every fixed $(k,R)$, both estimators are computable in arithmetic time
polynomial in $n,d$; the tail estimator is also polynomial in
$\log(1/\delta)$.
\end{theorem}

Throughout the paper, ``polynomial time'' means polynomially many exact real
arithmetic operations, as in the computational analysis of \citet{DWYZ23}.
We do not claim a Turing bit-complexity theorem; finite-precision issues and
the additional ingredients required for such a theorem are discussed in
\cref{sec:complexity}.

The estimator is proper: its output is itself a mixture of at most $k$
Gaussians with covariance $I_d$, not merely a density approximation or a
sampling oracle.  No eigengap, separation, or lower bound on the weights is
used.  The polynomial degree depends on $k$; the result does not claim
polynomial complexity when $k$ is part of the input.

For $k=2$, \citet{DWYZ23} already gave a different efficient estimator attaining
the sharp rate.  The construction here covers every fixed $k$, but it is a
complexity-theoretic rather than a practical algorithm: for $k=3$, the generic
proper-net search used below can require order $n^{58}$ candidates in the
worst case.  No attempt is made to optimize this exponent.

The expected-risk and tail estimators use the same range-finding and proper
fitting construction.  The former uses ordinary sample means and only their
second-moment bounds; the latter substitutes finite-covariance robust means.
This distinction is necessary because a robust mean estimator achieving
sub-Gaussian deviations under only finite covariance is generally allowed to
depend on the requested confidence level.

\section{Preliminaries}

\subsection{Symmetric tensors and latent moments}

For a random vector $V\in\R^d$, define
\[
 M_\ell(V)=\E V^{\otimes\ell},\qquad \ell\ge1.
\]
For a mixing distribution we also write $M_\ell(\Gamma)$.  Tensor spaces carry
the Frobenius inner product and norm.  The operator $\Sym$ is normalized
symmetrization, hence the orthogonal projection onto the symmetric tensor
subspace and therefore a contraction in Frobenius norm.  If $T$ is an
$\ell$-tensor and $q$ is an $r$-tensor with $r<\ell$, $T\contr q$ denotes
contraction of the last $r$ modes.

We use the following dimension-free result, which is Theorem~4.2 of
\citet{DWYZ23}.

\begin{theorem}[Moment characterization; Doss--Wu--Yang--Zhou]\label{thm:dwyz}
For two mixing distributions $\Gamma,\Gamma'$ supported on at most $k$ points
of $B_R^d$ and every $D\in\{H^2,\mathrm{KL},\chi^2\}$,
\begin{align}\label{eq:dwyz-both}
 &(C_Rk)^{-4k}
 \max_{1\le\ell\le2k-1}
 \fnorm{M_\ell(\Gamma)-M_\ell(\Gamma')}^2
 \le D(P_\Gamma,P_{\Gamma'}) \notag\\
 &\hspace{45mm}\le C_R e^{36k^2}
 \max_{1\le\ell\le2k-1}
 \fnorm{M_\ell(\Gamma)-M_\ell(\Gamma')}^2.
\end{align}
The constant does not depend on $d$.
\end{theorem}

In the convention of \citet{DWYZ23}, a $k$-atomic distribution is written
with $k$ weights that are allowed to vanish.  Thus \cref{thm:dwyz} applies to
two distributions with possibly different support sizes, provided each has
at most $k$ atoms.  In \cref{prop:quadratic-lb} below we apply it with the
theorem parameter set to $k=6$.

Only the Hellinger upper bound in \eqref{eq:dwyz-both} is needed for the main
result; the quadratic-sketch proposition later uses both sides and the
chi-squared upper bound.

\subsection{Hermite tensors}

Define the multivariate probabilists' Hermite tensors by the generating
function
\begin{equation}\label{eq:hermite-gen}
 \exp\{\ip{t}{x}-\norm{t}^2/2\}
 =\sum_{\ell=0}^\infty\frac1{\ell!}
   \ip{\cH_\ell(x)}{t^{\otimes\ell}}.
\end{equation}
If $Z\sim N(0,I_d)$, then
\begin{equation}\label{eq:hermite-unbiased}
 \E[\cH_\ell(u+Z)]=u^{\otimes\ell},\qquad
 \E\cH_\ell(X)=M_\ell(\Gamma).
\end{equation}
For $q\in\Sym^{\ell-1}(\R^d)$, define the one-free-index statistic
\begin{equation}\label{eq:fiber-stat}
 Y_{\ell,q}(x)=\cH_\ell(x)\contr q\in\R^d.
\end{equation}
Thus $\E Y_{\ell,q}(X)=M_\ell(\Gamma)\contr q$.

\subsection{A finite-covariance mean primitive}

We use a standard computational primitive.

\begin{proposition}[Polynomial-time sub-Gaussian mean estimation]\label{prop:robust-mean}
Let $Y_1,\ldots,Y_N$ be iid vectors in $\R^p$ with mean $g$ and finite
covariance $\Sigma$.  There is a polynomial-arithmetic-time estimator
$\widehat g$ such that, whenever $N\ge c\log(1/\eta)$,
\begin{equation}\label{eq:robust-mean}
 \norm{\widehat g-g}
 \le C\left(
  \sqrt{\frac{\tr\Sigma}{N}}+
  \sqrt{\frac{\norm{\Sigma}_{\op}\log(1/\eta)}{N}}
 \right)
\end{equation}
with probability at least $1-\eta$.
\end{proposition}

The proposition follows directly from Theorem~1.2 of \citet{Hop20}.  In that
result the confidence range is $\eta>2^{-N/C_2}$ for a universal $C_2$;
choosing the constant $c$ above sufficiently large makes
$N\ge c\log(1/\eta)$ imply this condition.  Its arithmetic runtime is
$O(Np)+(p\log(1/\eta))^{C_0}$ for a universal $C_0$, and neither $\Sigma$ nor
any moment beyond the covariance is required as input.  Theorem~1 of
\citet{CFB19} supplies a related faster implementation in the same confidence
regime.  Their procedure first compresses the $N$ observations into
$B=O(\log(1/\eta))$ block means and then applies an
$O(B^{3.5}+B^2p)$ optimization routine to those block means.  Including the
$O(Np)$ cost of forming them gives
$\widetilde O\{Np+p\log^2(1/\eta)+\log^{3.5}(1/\eta)\}$ arithmetic operations,
with an arbitrary numerical tolerance; tolerance $N^{-2}$ is absorbed in our
application.  When
$\log(1/\eta)$ is of order $N$ or larger, the main estimator instead uses its
trivial-output branch.

\section{The estimators}\label{sec:estimator}

Put $N=\lfloor n/3\rfloor$, split $3N$ observations into three independent
blocks $I_1,I_2,I_3$ of exactly $N$ observations each, and discard the at most
two remaining observations.  Set
\begin{equation}\label{eq:L-eps}
 L=2k-1,\qquad
 \eps=\sqrt{\frac{d+\log(1/\delta)}{n}}.
\end{equation}

We first describe the confidence-dependent estimator
$\widehat\Gamma^\delta$.
If $N<C_{k,R}^{(0)}(d+\log(1/\delta))$, for a sufficiently large fixed
constant $C_{k,R}^{(0)}$, return the one-atom mixing distribution
$\delta_0$.  Since then $\eps$ is bounded below by a constant, this case is
absorbed by \eqref{eq:main-tail}.  We henceforth describe the nontrivial case.
The threshold constant is chosen large enough that $\eps<1$ in this branch,
so the mesh $\tau_\delta=\min\{1,\eps\}$ below equals $\eps$.

\subsection{Block 1: A coarse range}
Form the unbiased latent second-moment estimator
\begin{equation}\label{eq:m2hat}
 \widetilde M_2=\frac1N\sum_{i\in I_1}(X_iX_i^\top-I_d).
\end{equation}
Let $A$ be a top-$q$ eigenspace of $\widetilde M_2$, where
$q=\min\{k,d\}$, and write $P_A$ for its projector.

\subsection{Block 2: Moment-fiber range augmentation}
For each $s=0,\ldots,L-1$, take a Frobenius-orthonormal basis
$\{E_{s,j}:1\le j\le D_s\}$ of $\Sym^s(A)$, where
\begin{equation}\label{eq:Ds}
 D_s=\binom{q+s-1}{s},\qquad D_0=1.
\end{equation}
Using block $I_2$ and \cref{prop:robust-mean}, estimate all vectors
\begin{equation}\label{eq:true-fibers}
 g_{s,j}=M_{s+1}(\Gamma)\contr E_{s,j}
 =\E\big[\cH_{s+1}(X)\contr E_{s,j}\big].
\end{equation}
Call the estimates $\widehat g_{s,j}$ and define
\begin{equation}\label{eq:H-def}
 H=\Span\left(A,\{\widehat g_{s,j}:0\le s\le L-1,
              1\le j\le D_s\}\right).
\end{equation}
The dimension is bounded by
\begin{equation}\label{eq:mk}
 m:=\dim H\le m_k:=k+\sum_{s=0}^{2k-2}\binom{k+s-1}{s}
 =k+\binom{3k-2}{2k-2}.
\end{equation}

\subsection{Block 3: Proper fitting in fixed dimension}
Let $B\in\R^{d\times m}$ have orthonormal columns spanning $H$.  From
$B^\top X_i$, $i\in I_3$, estimate
\begin{equation}\label{eq:low-moments}
 T_\ell=M_\ell(B^\top\#\Gamma)
 =\E\cH_\ell(B^\top X),\qquad 1\le\ell\le L,
\end{equation}
coordinatewise by scalar median-of-means; denote the estimates by
$\widetilde T_\ell$.

For the tail estimator, let
$\tau=\tau_\delta:=\min\{1,\eps\}$.  Construct a Euclidean $\tau$-net $\cC_\tau$ of
$B_R^m$ and an $\ell_1$-$\tau$-net $\cW_\tau$ of the simplex
$\Delta_{k-1}$.  Enumerate
\begin{equation}\label{eq:grid}
 \cG_\tau=
 \left\{\sum_{i=1}^k\bar w_i\delta_{\bar\theta_i}:
  (\bar w_1,\ldots,\bar w_k)\in\cW_\tau,
  \bar\theta_i\in\cC_\tau\right\},
\end{equation}
and select any minimizer
\begin{equation}\label{eq:minfit}
\widehat\gamma\in\arg\min_{\gamma\in\cG_\tau}
 \max_{1\le\ell\le L}\fnorm{M_\ell(\gamma)-\widetilde T_\ell}.
\end{equation}
For the statistical guarantee an exact minimizer is unnecessary: any
candidate whose displayed objective is within $c_{k,R}\tau$ of the minimum
obeys the same rate after changing constants; see \cref{lem:net-fit}.
Finally lift its atoms through $B$:
\begin{equation}\label{eq:lift}
 \widehat\Gamma=B\#\widehat\gamma.
\end{equation}
Because $B$ is an isometry from $\R^m$ to $H$, every lifted atom remains in
$B_R^d$.  Thus the output, which we denote by $\widehat\Gamma^\delta$, is
proper.

\subsection{Expected-risk version}
The expected-risk estimator $\widehat\Gamma^{\mathrm E}$ is defined without a
confidence-level input.  It uses the same three equal blocks as above.  If
$N<C_{k,R}^{(0)}d$, it returns $\delta_0$.  Otherwise, block 1 and the coarse
space $A$ are unchanged.  On block 2, each fiber $g_{s,j}$ is estimated by the
ordinary empirical mean of
$\cH_{s+1}(X)\contr E_{s,j}$, and $H$ is the span in \eqref{eq:H-def}.
On block 3, each $T_\ell$ is estimated by the ordinary empirical mean of
$\cH_\ell(B^\top X)$.  The proper fit \eqref{eq:minfit} is then performed on
the mesh
\[
 \tau_{\mathrm E}=\min\{1,\sqrt{d/n}\},
\]
and its atoms are lifted through $B$ exactly as in \eqref{eq:lift}.  The tail
estimator differs only in using the small-sample threshold
$C_{k,R}^{(0)}(d+\log(1/\delta))$, finite-covariance robust means for the
block-2 fibers, coordinate median-of-means for the block-3 moments, and the
mesh $\tau_\delta=\min\{1,\sqrt{(d+\log(1/\delta))/n}\}$.

All ambiguous choices of eigenspaces, orthonormal bases, and net minimizers
for either estimator are resolved by fixed Borel-measurable tie-breaking
rules, for example lexicographic Gram--Schmidt relative to the standard basis.

\section{The deterministic fiber-sketch lemma}

This section contains the main structural step.  It does not require $U$ to be
atomic.

\begin{lemma}[Fiber sketches control projection bias]\label{lem:fiber}
Let $U$ be any random vector supported in $B_R^d$, and let $A\subseteq H$ be
linear subspaces with projectors $P_A,P_H$.  For $s\ge0$, let
$\{E_{s,j}\}_{j=1}^{D_s}$ be a Frobenius-orthonormal basis of $\Sym^s(A)$,
and put
\[
 g_{s,j}=M_{s+1}(U)\contr E_{s,j}.
\]
Suppose vectors $\widehat g_{s,j}\in H$ are given.  Define
\begin{equation}\label{eq:fiber-errors}
 e_s^2=\sum_{j=1}^{D_s}\norm{\widehat g_{s,j}-g_{s,j}}^2,
 \qquad
 \rho=\E\norm{(I-P_A)U}^2.
\end{equation}
Then
\begin{equation}\label{eq:first-moment-proj}
 \norm{M_1(U)-M_1(P_HU)}\le e_0,
\end{equation}
and, for every $\ell\ge2$,
\begin{equation}\label{eq:fiber-bound}
 \fnorm{M_\ell(U)-M_\ell(P_HU)}
 \le \ell e_{\ell-1}+3^\ell R^{\ell-2}\rho.
\end{equation}
\end{lemma}

\begin{proof}
Write the orthogonal decomposition
\[
 a=P_AU,\qquad b=(P_H-P_A)U,\qquad c=(I-P_H)U,
 \qquad r=b+c=(I-P_A)U.
\]
For $s\ge0$, define
\[
 F_s=\E[c\otimes a^{\otimes s}]
 \in H^\perp\otimes\Sym^s(A).
\]
Since $E_{s,j}$ lies entirely in $A^{\otimes s}$,
\[
 F_s\contr E_{s,j}=(I-P_H)g_{s,j}.
\]
Parseval in the last $s$ modes and the fact that
$\widehat g_{s,j}\in H$ give
\begin{equation}\label{eq:parseval-fiber}
 \fnorm{F_s}^2
 =\sum_{j=1}^{D_s}\norm{(I-P_H)g_{s,j}}^2
 \le\sum_{j=1}^{D_s}\norm{g_{s,j}-\widehat g_{s,j}}^2=e_s^2.
\end{equation}

Now expand and symmetrize
\begin{align}\label{eq:tensor-expansion}
 &(a+b+c)^{\otimes\ell}-(a+b)^{\otimes\ell}\notag\\
 &\quad=
 \sum_{\substack{p+q+t=\ell\\t\ge1}}
 \binom{\ell}{p,q,t}
 \Sym\big(a^{\otimes p}\otimes b^{\otimes q}
                    \otimes c^{\otimes t}\big).
\end{align}
The unique summand with exactly one factor outside $A$ is
$\ell\Sym(c\otimes a^{\otimes(\ell-1)})$.  After expectation, its norm is
at most $\ell\fnorm{F_{\ell-1}}\le\ell e_{\ell-1}$ by
\eqref{eq:parseval-fiber} and the contractivity of $\Sym$.

Every other summand in \eqref{eq:tensor-expansion} has $q+t\ge2$.  Pointwise,
using orthogonality and $\norm{U}\le R$,
\[
 \norm{a}^{p}\norm{b}^{q}\norm{c}^{t}
 \le R^{\ell-2}\norm{r}^2.
\]
The absolute multinomial coefficients sum to at most $3^\ell$.  Taking
expectations proves \eqref{eq:fiber-bound}.  The case $\ell=1$ is
\eqref{eq:parseval-fiber} with $s=0$.
\end{proof}

If $a_1,\ldots,a_q$ is an orthonormal basis of $A$, a normalized monomial
basis used in \cref{lem:fiber} is
\begin{equation}\label{eq:sym-basis}
 E_\alpha=\sqrt{\frac{s!}{\alpha!}}
 \Sym(a_1^{\otimes\alpha_1}\otimes\cdots\otimes
      a_q^{\otimes\alpha_q}),\qquad |\alpha|=s.
\end{equation}
The factor in \eqref{eq:sym-basis} is needed for the exact Parseval identity.

\section{Statistical control of the range}

\subsection{Residual second-moment energy}

Let $M_2=\E UU^\top$.  It is positive semidefinite and has rank at most $k$.

\begin{lemma}[Gap-free coarse range]\label{lem:coarse}
Put $\zeta=\norm{\widetilde M_2-M_2}_{\op}$.  For the space $A$ in block 1,
\begin{equation}\label{eq:rho-kyfan}
 \rho:=\E\norm{(I-P_A)U}^2
 =\tr((I-P_A)M_2)\le2k\zeta.
\end{equation}
Moreover, in the nontrivial regime of \cref{thm:main}, with probability at
least $1-\delta/4$,
\begin{equation}\label{eq:zeta-rate}
 \zeta\le C_{k,R}\eps.
\end{equation}
\end{lemma}

\begin{proof}
Let $P_*$ project onto $\operatorname{range}(M_2)$, extended arbitrarily to
rank $q$ if necessary.  By optimality of the top-$q$ eigenspace,
$\tr(P_A\widetilde M_2)\ge\tr(P_*\widetilde M_2)$.  Hence
\begin{align*}
 \tr((I-P_A)M_2)
 &=\tr(P_*M_2)-\tr(P_AM_2)\\
 &\le \tr(P_*(M_2-\widetilde M_2))
     +\tr(P_A(\widetilde M_2-M_2))
 \le2k\zeta.
\end{align*}
This argument uses no eigengap.

Lemma~3.6 of \citet{DWYZ23}, applied to the same raw estimator
\eqref{eq:m2hat}, gives
\[
 \zeta\le C_R\left(
  \sqrt{\frac dN}+k\sqrt{\frac{\log(k/\delta)}N}
  +\frac{\log(1/\delta)}N\right).
\]
The nontrivial branch takes its fixed constant sufficiently large that
$N>d$.  Since $k$ is fixed and $N\gtrsim d+\log(1/\delta)$, the display is at
most $C_{k,R}\eps$.
\end{proof}

\subsection{Sharp estimation of the Hermite fibers}

\begin{lemma}[Fiber covariance]\label{lem:fiber-cov}
Fix $\ell\ge1$, a symmetric tensor $q\in\Sym^{\ell-1}(\R^d)$ with
$\fnorm q=1$, and let $Y_{\ell,q}$ be as in \eqref{eq:fiber-stat}.  Uniformly
over all mixing distributions supported in $B_R^d$,
\begin{equation}\label{eq:fiber-cov-bound}
 \norm{\operatorname{Cov}(Y_{\ell,q}(X))}_{\op}\le C_{\ell,R},
 \qquad
 \tr\operatorname{Cov}(Y_{\ell,q}(X))\le C_{\ell,R}d.
\end{equation}
\end{lemma}

\begin{proof}
Fix a unit vector $v$ and put $S=\Sym(v\otimes q)$, so $\fnorm S\le1$ and
$\ip{v}{Y_{\ell,q}(x)}=\ip{\cH_\ell(x)}S$.  The Hermite translation formula
gives
\[
 \ip{\cH_\ell(u+Z)}S
 =\sum_{j=0}^{\ell}\binom{\ell}{j}
   \ip{\cH_j(Z)}{S\contr u^{\otimes(\ell-j)}}.
\]
Because $S$ is symmetric, every contraction
$S\contr u^{\otimes(\ell-j)}$ is already a symmetric $j$-tensor.  Thus the
Gaussian-chaos isometry and orthogonality of distinct chaoses give the exact
equality, for fixed $u$,
\begin{align}\label{eq:chaos-isometry}
 \E_Z\ip{v}{Y_{\ell,q}(u+Z)}^2
 &=\sum_{j=0}^{\ell}\binom{\ell}{j}^2j!
   \fnorm{S\contr u^{\otimes(\ell-j)}}^2\\
 &\le \sum_{j=0}^{\ell}\binom{\ell}{j}^2j!R^{2(\ell-j)}
 =:C_{\ell,R}.\notag
\end{align}
Integrating over $U$ and subtracting the squared mean proves the operator
bound.  Summing the directional variance over an orthonormal basis of
$\R^d$ proves the trace bound.
\end{proof}

The use of a finite-covariance robust mean is material.  A raw empirical mean
of a degree-$\ell$ Hermite polynomial generally has sub-Weibull rather than
sub-Gaussian tails and need not yield the additive
$\sqrt{d/N}+\sqrt{\log(1/\delta)/N}$ guarantee.

\begin{proposition}[Range recovery in the moment metric]\label{prop:range}
With probability at least $1-\delta/2$ over blocks 1 and 2,
\begin{equation}\label{eq:projection-moment-rate}
 \max_{1\le\ell\le L}
 \fnorm{M_\ell(\Gamma)-P_H^{\otimes\ell}M_\ell(\Gamma)}
 \le C_{k,R}\eps.
\end{equation}
\end{proposition}

\begin{proof}
Condition on block 1.  Then $A$ and all basis tensors $E_{s,j}$ are fixed and
independent of block 2.  Apply \cref{prop:robust-mean,lem:fiber-cov} with
failure probability $\delta/(4J_k^\star)$ to each fiber, where
\begin{equation}\label{eq:Jk}
 J(A):=\sum_{s=0}^{2k-2}D_s
 \le J_k^\star:=\binom{3k-2}{2k-2}.
\end{equation}
Since $J_k^\star$ is constant for fixed $k$, with conditional probability at least
$1-\delta/4$,
\[
 \max_{s,j}\norm{\widehat g_{s,j}-g_{s,j}}
 \le C_{k,R}\sqrt{\frac{d+\log(1/\delta)}N}
 \le C_{k,R}\eps.
\]
Thus every aggregate error $e_s$ in \cref{lem:fiber} is at most
$C_{k,R}\eps$.  On the event in \cref{lem:coarse}, $\rho\le C_{k,R}\eps$.
For every $1\le\ell\le L=2k-1$, the required index satisfies
$\ell-1\le L-1=2k-2$, exactly the range of fibers constructed in block 2.
Intersect the coarse and all-fiber events.  On their intersection,
\cref{lem:fiber} applies simultaneously to every $1\le\ell\le L$.
\end{proof}

\subsection{Linear-time evaluation of a fiber observation}

Although \eqref{eq:fiber-stat} is written tensorially, no ambient tensor is
formed.  Let $V\in\R^{d\times q}$ have orthonormal columns spanning $A$, put
$a=V^\top x$ and $Q=I-VV^\top$, and let
$q_0\in\Sym^{\ell-1}(\R^q)$.  Orthogonal product factorization of Hermite
tensors gives the exact identity
\begin{equation}\label{eq:fiber-computation}
 Y_{\ell,V^{\otimes(\ell-1)}q_0}(x)
 =Qx\,
   \ip{\cH_{\ell-1}^{(q)}(a)}{q_0}
 +V\big[\cH_\ell^{(q)}(a)\contr q_0\big].
\end{equation}
Here is a direct verification, including the normalization.  Decompose the
Hermite generating function along $A\oplus A^\perp$:
\begin{align*}
 &\exp\{\ip{t}{x}-\norm t^2/2\}\\
 &\quad=
 \exp\{\ip{V^\top t}{a}-\norm{V^\top t}^2/2\}
 \exp\{\ip{Qt}{Qx}-\norm{Qt}^2/2\}.
\end{align*}
Test the degree-$\ell$ coefficient against one free direction $h$ and the
last $\ell-1$ directions encoded by $V^{\otimes(\ell-1)}q_0$.  If
$h\in A^\perp$, the only contributing product has degree $1$ in the second
factor and degree $\ell-1$ in the first, yielding
$\ip{h}{Qx}\ip{\cH_{\ell-1}^{(q)}(a)}{q_0}$.  If $h=Vz\in A$, all $\ell$
directions come from the first factor, yielding
$\ip{z}{\cH_\ell^{(q)}(a)\contr q_0}$.  The binomial coefficient in the
degree split cancels the $1/\ell$ from normalized symmetrization of the one
free slot.  These two orthogonal components prove
\eqref{eq:fiber-computation} with no missing combinatorial constant.

The first term is one scalar times an ambient vector; the second lies in a
space of dimension at most $k$.  For fixed $k$ and $\ell\le2k-1$, it takes
$O_k(d)$ arithmetic operations to evaluate \eqref{eq:fiber-computation}.

\section{Proper fitting after projection}

Let $\Gamma_H=(B^\top)\#\Gamma$, a mixing distribution on $B_R^m$.  Conditional
on the first two blocks, block 3 is independent and $B^\top Z\sim N(0,I_m)$.

\begin{lemma}[Low-dimensional moment estimation]\label{lem:low-mom}
With conditional probability at least $1-\delta/4$,
\begin{equation}\label{eq:low-mom-rate}
 \max_{1\le\ell\le L}
 \fnorm{\widetilde T_\ell-M_\ell(\Gamma_H)}
 \le C_{k,R}\sqrt{\frac{1+\log(1/\delta)}N}
 \le C_{k,R}\eps.
\end{equation}
\end{lemma}

\begin{proof}
For each $\ell$, vectorize $\cH_\ell(B^\top X)$ in a Frobenius-orthonormal
basis of $\Sym^\ell(\R^m)$.  The number of coordinates is
$\binom{m+\ell-1}{\ell}=O_k(1)$.  The calculation in
\eqref{eq:chaos-isometry} bounds each coordinate variance by $C_{k,R}$.
Apply scalar median-of-means to each coordinate with the failure probability
divided over the $O_k(1)$ coordinates and degrees.  Combining coordinate
bounds in Frobenius norm proves the claim.
\end{proof}

\begin{lemma}[Proper net fit]\label{lem:net-fit}
On the event \eqref{eq:low-mom-rate}, the minimizer in \eqref{eq:minfit}
satisfies
\begin{equation}\label{eq:fit-moment-rate}
 \max_{1\le\ell\le L}
 \fnorm{M_\ell(\widehat\gamma)-M_\ell(\Gamma_H)}
 \le C_{k,R}(\tau+\eps).
\end{equation}
Moreover,
\begin{equation}\label{eq:net-size}
 |\cG_\tau|\le(C_R/\tau)^{km+k-1}.
\end{equation}
More generally, if a candidate has objective value in \eqref{eq:minfit}
within $\xi_{\rm fit}$ of the minimum, the right side of
\eqref{eq:fit-moment-rate} increases by at most $\xi_{\rm fit}$.
\end{lemma}

\begin{proof}
Pad $\Gamma_H$ with zero-weight atoms to write it as
$\sum_{i=1}^kw_i\delta_{\theta_i}$.  Choose net points with
$\norm{\theta_i-\bar\theta_i}\le\tau$ and weights with
$\norm{w-\bar w}_1\le\tau$.  For every $\ell\le L$,
\begin{align*}
 \fnorm{M_\ell(\Gamma_H)-\sum_i\bar w_i\bar\theta_i^{\otimes\ell}}
 &\le R^\ell\norm{w-\bar w}_1
  +\sum_i\bar w_i
    \fnorm{\theta_i^{\otimes\ell}-\bar\theta_i^{\otimes\ell}}\\
 &\le R^\ell\tau+\ell R^{\ell-1}\tau
 \le C_{k,R}\tau,
\end{align*}
where we used the telescoping inequality
\[
 \fnorm{x^{\otimes\ell}-y^{\otimes\ell}}
 \le\ell\max\{\norm x,\norm y\}^{\ell-1}\norm{x-y}.
\]
Call the resulting candidate $\gamma_0$.  If $S(\gamma)$ denotes the objective
in \eqref{eq:minfit}, then a $\xi_{\rm fit}$-approximate minimizer satisfies
$S(\widehat\gamma)\le S(\gamma_0)+\xi_{\rm fit}$.  Two applications of the
triangle inequality and \eqref{eq:low-mom-rate} therefore give
\eqref{eq:fit-moment-rate} with the additional term $\xi_{\rm fit}$.  The
stated exact-minimizer bound is the special case $\xi_{\rm fit}=0$.

A Euclidean ball in fixed dimension $m$ has a $\tau$-net of size at most
$(C_R/\tau)^m$, and the simplex has an $\ell_1$-$\tau$-net of size at most
$(C/\tau)^{k-1}$.  Taking $k$ ordered location points proves
\eqref{eq:net-size}.
\end{proof}

\section{Proof of the main theorem}

\begin{proof}
We first prove the tail statement and denote the estimator constructed in
\cref{eq:m2hat,eq:true-fibers,eq:minfit,eq:lift} by
$\widehat\Gamma^\delta$.  Work in the nontrivial regime and intersect the events in
\cref{prop:range,lem:low-mom}.  The former has failure probability at most
$\delta/2$ and the latter at most $\delta/4$, so their intersection has
probability at least $1-3\delta/4\ge1-\delta$.  Because $B$ is an isometry,
\begin{align*}
 &\fnorm{M_\ell(P_H\#\Gamma)-M_\ell(B\#\widehat\gamma)}\\
 &\qquad=
 \fnorm{B^{\otimes\ell}
   \{M_\ell(\Gamma_H)-M_\ell(\widehat\gamma)\}}
 =\fnorm{M_\ell(\Gamma_H)-M_\ell(\widehat\gamma)}.
\end{align*}
Combining \eqref{eq:projection-moment-rate},
\eqref{eq:fit-moment-rate}, and $\tau=\eps$ in the nontrivial regime yields
\begin{equation}\label{eq:all-moments-final}
 \max_{1\le\ell\le2k-1}
 \fnorm{M_\ell(\Gamma)-M_\ell(\widehat\Gamma^\delta)}
 \le C_{k,R}\eps.
\end{equation}
Both measures have at most $k$ atoms in $B_R^d$.  Apply
\cref{thm:dwyz} to \eqref{eq:all-moments-final} to obtain
$H(P_{\widehat\Gamma^\delta},P_\Gamma)\le C_{k,R}\eps$.

If the nontrivial sample-size condition fails, the estimator returns
$\delta_0$.  Since Hellinger distance is at most $\sqrt2$ and $\eps$ is then
bounded below, the same conclusion holds after changing $C_{k,R}$.  Adjusting
the constant proves \eqref{eq:main-tail}.

We next analyze the single estimator $\widehat\Gamma^{\mathrm E}$ defined in
\cref{sec:estimator} for \eqref{eq:main-expectation}.  Put
$\eps_{\mathrm E}=\sqrt{d/n}$ and work first in its nontrivial regime
$N\ge C_{k,R}^{(0)}d$.

To spell out the integration step, write $u=\log(1/\delta)$.  Lemma~3.6 of
\citet{DWYZ23} implies, for $u\ge1$, that with probability at least $1-e^{-u}$,
\[
 \norm{\widetilde M_2-M_2}_{\op}
 \le C_R\left\{\sqrt{\frac dN}
   +k\sqrt{\frac{\log k+u}{N}}+\frac uN\right\}=:f(u).
\]
Since $f$ is increasing, the tail-integral identity yields
$\E\norm{\widetilde M_2-M_2}_{\op}
\le f(1)+\int_1^\infty e^{-u}f'(u)\,du$.  The elementary integrals of
$e^{-u}u^{-1/2}$ and $e^{-u}$ therefore give
\[
 \E\norm{\widetilde M_2-M_2}_{\op}
 \le C_{k,R}\left(\sqrt{\frac dN}+\frac dN\right),
\]
and therefore, since $N\gtrsim_{k,R}d$,
\begin{equation}\label{eq:expected-zeta}
 \E\norm{\widetilde M_2-M_2}_{\op}
 \le C_{k,R}\sqrt{\frac dN}.
\end{equation}
Conditional on block 1, \cref{lem:fiber-cov} and the iid empirical-mean
identity give, for every fiber,
\begin{equation}\label{eq:expected-fiber}
 \E\left[\norm{\widehat g_{s,j}-g_{s,j}}^2\mid I_1\right]
 =\frac{\tr\operatorname{Cov}(Y_{s+1,E_{s,j}}(X))}{N}
 \le C_{k,R}\frac dN.
\end{equation}
For the aggregate error $e_s$ in \cref{lem:fiber}, conditional Jensen gives
\begin{align*}
 \E[e_s\mid I_1]
 &\le\left(\sum_{j=1}^{D_s}
       \E[\norm{\widehat g_{s,j}-g_{s,j}}^2\mid I_1]
       \right)^{1/2}
 \le C_{k,R}\sqrt{\frac dN}.
\end{align*}
Dependence among the fibers is irrelevant.  Samplewise,
\cref{lem:coarse,lem:fiber} gives
\[
 \max_{\ell\le L}
 \fnorm{M_\ell(\Gamma)-P_H^{\otimes\ell}M_\ell(\Gamma)}
 \le L\max_{s<L}e_s+C_{k,R}\rho,
 \qquad \rho\le2k\norm{\widetilde M_2-M_2}_{\op}.
\]
There are only $O_k(1)$ fibers and degrees.  Taking expectations and using
\eqref{eq:expected-zeta}--\eqref{eq:expected-fiber}, together with
$N\asymp n$, implies
\begin{equation}\label{eq:expected-projection}
 \E\max_{1\le\ell\le L}
 \fnorm{M_\ell(\Gamma)-P_H^{\otimes\ell}M_\ell(\Gamma)}
 \le C_{k,R}\sqrt{\frac dn}.
\end{equation}

Condition now on $(I_1,I_2)$.  Then $B,m$, and $\Gamma_H$ are fixed and the
third block is independent.  In Frobenius-orthonormal coordinates, the total
number of coordinates in
$\bigoplus_{\ell=1}^L\Sym^\ell(\R^m)$ is $O_k(1)$ and each coordinate has
variance $O_{k,R}(1)$.  Hence
\begin{align*}
 &\E\left[
   \max_{\ell\le L}\fnorm{\widetilde T_\ell-M_\ell(\Gamma_H)}
   \mid I_1,I_2\right]\\
 &\quad\le
 \left(\sum_{\ell=1}^L
 \E[\fnorm{\widetilde T_\ell-M_\ell(\Gamma_H)}^2\mid I_1,I_2]
 \right)^{1/2}
 \le C_{k,R}N^{-1/2}.
\end{align*}
The bound is uniform in the conditioning, so
\begin{equation}\label{eq:expected-low-mom}
 \E\max_{1\le\ell\le L}
 \fnorm{\widetilde T_\ell-M_\ell(\Gamma_H)}
 \le C_{k,R}N^{-1/2}.
\end{equation}
The proof of \cref{lem:net-fit} is deterministic conditional on the estimated
moments and gives the sharper random bound
\[
 \max_{\ell\le L}
 \fnorm{M_\ell(\widehat\gamma)-M_\ell(\Gamma_H)}
 \le C_{k,R}\tau_{\mathrm E}+
 2\max_{\ell\le L}\fnorm{\widetilde T_\ell-M_\ell(\Gamma_H)}.
\]
Combining this display with \eqref{eq:expected-projection}--
\eqref{eq:expected-low-mom}, using $N^{-1/2}\le\sqrt{d/N}\lesssim
\sqrt{d/n}$ for $d\ge1$, lifting by $B$, and applying the pointwise Hellinger
upper bound in \cref{thm:dwyz} proves \eqref{eq:main-expectation}.  If
$N<C_{k,R}^{(0)}d$, the estimator returns $\delta_0$; then
$\min\{1,\sqrt{d/n}\}$ is bounded below by a constant depending only on
$k,R$, so $H\le\sqrt2$ proves the same claim.

Polynomial arithmetic time for both versions is verified in
\cref{sec:complexity}.
\end{proof}

\section{Arithmetic complexity and numerical issues}\label{sec:complexity}

The total number of fibers is $J(A)$ and is bounded by
\[
 J_k^\star=\binom{3k-2}{2k-2},
\]
and $m\le m_k=k+J_k^\star$.  These are constants for fixed $k$.  Equality
$J(A)=J_k^\star$ holds only when $q=k$; the upper bound is all that is used.
Forming
\eqref{eq:m2hat} and a full eigendecomposition costs at most
$O(nd^2+d^3)$ arithmetic operations.  The robust mean calls are polynomial
time by \cref{prop:robust-mean}; the fiber observations themselves are
computed in $O_k(d)$ time via \eqref{eq:fiber-computation}.  For either
estimator, let $\tau_*$ denote its relevant mesh, $\tau_\delta$ or
$\tau_{\mathrm E}$.  The final search has size
\[
 (C_R/\tau_*)^{km_k+k-1}.
\]
In the nontrivial regime $\tau_*\gtrsim n^{-1/2}$, hence this is polynomial in
$n$ for every fixed $k$.  One coarse bound is
\begin{equation}\label{eq:runtime}
 O(nd^2+d^3)+\operatorname{poly}_{k,R}(n,d,\log(1/\delta))
 +(C_R/\tau_*)^{km_k+k-1}
\end{equation}
for the tail estimator; for the expected-risk estimator, omit the
$\log(1/\delta)$ input and take $\tau_*=\tau_{\mathrm E}$.
For example, $m_3\le38$ and the generic net exponent is
$3\cdot38+2=116$.  Since the smallest relevant mesh is of order
$n^{-1/2}$, this crude bound permits as many as $O_R(n^{58})$ candidates for
$k=3$.  The construction proves polynomial-time computability for fixed $k$;
it is not proposed as a practical implementation, and the exponent is not
optimized.

The enumeration also need not resolve an exact comparison between nearly tied
scores.  By \cref{lem:net-fit}, it may stop at any candidate whose score is
within $c_{k,R}\tau_*$ of the best enumerated score.  Thus objective values
only need to be evaluated to additive accuracy $O_{k,R}(\tau_*)$.

The nets can be enumerated explicitly.  Write $\tau$ for the relevant mesh in
the rest of this paragraph.  If $R\le\tau$, the origin alone covers $B_R^m$.
Otherwise take a Cartesian mesh of spacing $h\asymp\tau/\sqrt m$ inside
$B_R^m$.  To cover a point on the boundary,
replace $\theta$ by $(1-\tau/(4R))\theta$ and then round each coordinate; with
a sufficiently small absolute constant in $h$, the rounded point remains in
$B_R^m$ and is within distance $\tau$.  For the
weights, take all integer compositions of
$Q=\lceil2k/\tau\rceil$ divided by $Q$.  Standard balanced rounding gives an
$\ell_1$ error at most $k/Q\le\tau/2$.  These constructions have the
cardinality asserted in \eqref{eq:net-size} and can be generated with
polynomial delay for fixed $k$.

Two details avoid hidden conditioning assumptions.

\begin{enumerate}
\item No eigengap is needed.  In finite precision it suffices to compute a
rank-$q$ projector whose Ky Fan objective is within $\xi$ of optimum.  The
proof of \cref{lem:coarse} then gives $\rho\le2k\zeta+\xi$.

\item The real-arithmetic theorem uses the full algebraic span of all estimated
fibers, as in \eqref{eq:H-def}; no threshold, eigengap, or conditioning
decision is needed.  If a thresholded SVD is
used in a finite-precision implementation, \cref{lem:fiber} should be written
with
\[
 e_s^2=\sum_j\operatorname{dist}(g_{s,j},H)^2
 \]
instead.  Then
\[
 \operatorname{dist}(g_{s,j},H)
 \le\norm{g_{s,j}-\widehat g_{s,j}}
    +\operatorname{dist}(\widehat g_{s,j},H).
\]
Discarding residual columns at singular-value threshold $\xi$ therefore adds
at most $O(\sqrt{D_s}\,\xi)=O_k(\xi)$ to the aggregate error.  One must not
discard a column while continuing to set its distance term to zero.
\end{enumerate}

Thus standard polynomial-accuracy eigensolvers and SVD routines preserve the
statistical bound.  The theorem is stated in the real-arithmetic model.  A
rational implementation of the final search can use a dyadic location mesh
of spacing $\asymp\tau$ and a simplex grid with denominator $O_k(1/\tau)$;
these grid points have $O_k(\log n)$ bits.  A full Turing-model treatment would
also specify weak-SDP precision for the chosen robust-mean primitive, but no
statistical issue is hidden there.

\section{Why a quadratic range finder loses a square root}

The original general algorithm controls projection using
\[
 H^2(P_\Gamma,P_{P_A\#\Gamma})
 \le\tfrac12W_2^2(\Gamma,P_A\#\Gamma)
 \lesssim_k\norm{\widetilde M_2-M_2}_{\op}.
\]
Since the operator-norm error is $\sqrt{d/n}$, this gives
$H\lesssim(d/n)^{1/4}$.  The following proposition isolates a local
fourth-root obstruction for methods that discard the signed samples and retain
only quadratic sketches.  It does not apply to our algorithm, which uses fresh
raw observations after the coarse quadratic range-finding step.  Over the full
class, quadratic sketches also have a stronger global sign nonidentifiability;
the cap construction below deliberately removes that artifact.

For $R>0$ and $d\ge3$, fix an orthonormal basis $e_1,\ldots,e_d$, put
$a=R/4$, and define the
sign-oriented cap
\[
 \mathcal V_d^+=\left\{\frac{\sqrt3}{2}e_2+\frac12z:
 z\in S(\Span\{e_3,\ldots,e_d\})\right\}.
\]
For $0<t\le t_0:=\min\{R/4,1\}$, let
\begin{equation}\label{eq:local-cap-family}
 \mathcal C_{t,d}(R)=\left\{\Gamma_{t,v}:v\in\mathcal V_d^+\right\},
 \qquad
 \Gamma_{t,v}=\frac13\left(
  \delta_{-ae_1+tv}+\delta_{-2tv}+\delta_{ae_1+tv}\right).
\end{equation}
Every member has three atoms in $B_R^d$, and the cap fixes the sign because
$\ip{v}{e_2}=\sqrt3/2$ for every $v\in\mathcal V_d^+$.

\begin{proposition}[Local fourth-root obstruction]
\label{prop:quadratic-lb}
Fix $R>0$.  There are constants $\alpha_R,c_R,c_0>0$ and $d_0$ such that,
whenever $d_0\le d\le n$, setting $t=\alpha_R(d/n)^{1/4}$ makes
$\mathcal C_{t,d}(R)$ well defined, and every possibly randomized density
estimator $\widehat P$ measurable with respect to
$(X_1X_1^\top,\ldots,X_nX_n^\top)$ satisfies
\begin{equation}\label{eq:quadratic-lb}
 \sup_{\Gamma\in\mathcal C_{t,d}(R)}
 \Pp_\Gamma\left\{
 H(\widehat P,P_\Gamma)\ge c_R(d/n)^{1/4}
 \right\}\ge c_0.
\end{equation}
Moreover, for
$\Gamma_0=\frac13(\delta_{-ae_1}+\delta_0+\delta_{ae_1})$,
\begin{equation}\label{eq:local-cap-radius}
 \sup_{\Gamma\in\mathcal C_{t,d}(R)}H(P_\Gamma,P_{\Gamma_0})\le C_Rt.
\end{equation}
\end{proposition}

\begin{proof}
Each distribution in \eqref{eq:local-cap-family} is centered and
\begin{equation}\label{eq:diagnostic-moments}
 M_2(\Gamma_{t,v})=\frac{2a^2}{3}e_1e_1^\top+2t^2vv^\top,
 \qquad
 M_3(\Gamma_{t,v})[e_1,e_1,\cdot]=\frac{2a^2t}{3}v.
\end{equation}

A standard spherical packing inside $\mathcal V_d^+$ gives
$v_1,\ldots,v_M$, with
$M\ge\exp(c d)$ and $\norm{v_i-v_j}\ge1/2$ for $i\ne j$.

Let $T(x)=xx^\top$ and sign-symmetrize
\[
 \bar\Gamma_{t,v}=\frac12\left(
  \Gamma_{t,v}+(-\operatorname{id})\#\Gamma_{t,v}\right).
\]
Because $T(x)=T(-x)$ and Gaussian noise is symmetric,
\begin{equation}\label{eq:quad-sym-law}
 T\#P_{\Gamma_{t,v}}=T\#P_{\bar\Gamma_{t,v}}.
\end{equation}
The symmetrized mixing distribution has at most six atoms.  Its odd moments
vanish.  For $r\ge1$, direct binomial expansion gives
\begin{align}\label{eq:sym-even-expansion}
 M_{2r}(\bar\Gamma_{t,v})
 &=\frac23\sum_{j=0}^{r}\binom{2r}{2j}
   a^{2r-2j}t^{2j}
   \Sym\!\left(e_1^{\otimes(2r-2j)}\otimes v^{\otimes2j}\right)
   +\frac{2^{2r}}3t^{2r}v^{\otimes2r}.
\end{align}
The $j=0$ term is independent of $v$, while every other term contains at
least two factors of $tv$.  Moreover, for unit vectors $v,w$ and every
$m\ge1$, telescoping gives
$\fnorm{v^{\otimes m}-w^{\otimes m}}\le m\norm{v-w}$.  Since $t\le1$,
\eqref{eq:sym-even-expansion} therefore implies, uniformly over $v,w$ and
$\ell\le11$,
\begin{equation}\label{eq:sym-moment-t2}
 \fnorm{M_\ell(\bar\Gamma_{t,v})-
         M_\ell(\bar\Gamma_{t,w})}\le C_{\ell,R}t^2.
\end{equation}
The chi-squared upper bound in Theorem~4.2 of \citet{DWYZ23}, applied with its
atom parameter set to $6$, uses moments only through degree $11$ and implies
\begin{equation}\label{eq:sym-chi}
 \chi^2(P_{\bar\Gamma_{t,v}},P_{\bar\Gamma_{t,w}})\le C_Rt^4.
\end{equation}
If $\mathbb Q_v$ is the law of all $n$ quadratic sketches, KL tensorization,
\eqref{eq:quad-sym-law}, $\mathrm{KL}\le\chi^2$, and data processing for
$\chi^2$ give
\begin{align}\label{eq:quad-kl}
 \mathrm{KL}(\mathbb Q_v\|\mathbb Q_w)
 &=n\,\mathrm{KL}\bigl(T\#P_{\bar\Gamma_{t,v}}
                 \|T\#P_{\bar\Gamma_{t,w}}\bigr)\notag\\
 &\le n\,\chi^2\bigl(T\#P_{\bar\Gamma_{t,v}}
                 \|T\#P_{\bar\Gamma_{t,w}}\bigr)\notag\\
 &\le n\,\chi^2\bigl(P_{\bar\Gamma_{t,v}}
                 \|P_{\bar\Gamma_{t,w}}\bigr)
 \le C_Rnt^4.
\end{align}

In the original, unsymmetrized experiment, the contraction in
\eqref{eq:diagnostic-moments} gives
\[
 \fnorm{M_3(\Gamma_{t,v})-M_3(\Gamma_{t,w})}
 \ge\frac{2a^2t}{3}\norm{v-w}.
\]
The Hellinger lower bound in the same moment characterization, now with three
atoms, yields
\begin{equation}\label{eq:target-separation}
 H(P_{\Gamma_{t,v}},P_{\Gamma_{t,w}})
 \ge c_Rt\norm{v-w}\ge c_Rt.
\end{equation}

Take $t=\alpha_R(d/n)^{1/4}$, with $\alpha_R$ sufficiently small that
$t\le t_0$ and $C_Rnt^4\le(\log M)/8$.  Fano's inequality and
\eqref{eq:quad-kl} give a constant lower bound on the probability of failing
to identify the packing index.  For completeness, given an arbitrary density
output $\widehat P$, define the nearest-neighbour decoder
\[
 \widehat J\in\arg\min_{1\le j\le M}H(\widehat P,P_{\Gamma_{t,v_j}}),
\]
with deterministic tie-breaking.  If the true index is $J$ and
\[
 H(\widehat P,P_{\Gamma_{t,v_J}})
 <\frac12\min_{i\ne J}
 H(P_{\Gamma_{t,v_i}},P_{\Gamma_{t,v_J}}),
\]
the triangle inequality forces $\widehat J=J$.  Thus the Fano lower bound for
$\Pp\{\widehat J\ne J\}$ and \eqref{eq:target-separation} imply
\eqref{eq:quadratic-lb} for every density-valued estimator, not only estimators
whose output belongs to the packing family.

Finally, for every $\ell\le5$, direct expansion of
\eqref{eq:local-cap-family} gives
$\fnorm{M_\ell(\Gamma_{t,v})-M_\ell(\Gamma_0)}\le C_{\ell,R}t$ uniformly in
$v$.  The Hellinger upper bound in Theorem~4.2 of \citet{DWYZ23}, now with
atom parameter $3$, proves \eqref{eq:local-cap-radius}.
\end{proof}

Over the unrestricted three-atomic class, a stronger constant lower bound
follows immediately from the exact sign aliasing between
$\Gamma_{t_0,v}$ and $\Gamma_{t_0,-v}$.  The proposition is deliberately
stated on the sign-oriented local cap to remove that artifact and isolate the
$t^2$-versus-$t$ mechanism responsible for the fourth root.  It is not a lower
bound for the full algorithm of \citet{DWYZ23}, which uses the original signed
observations after constructing its projection; it diagnoses only procedures
whose retained information is measurable with respect to the quadratic
sketches.

\section{Relation to recent algorithms}

The result targets a sharper conjunction of properties than generic PAC
learnability: fixed $k$, arbitrary ambient dimension, no separation or
minimum-weight condition, proper output, polynomial time, and the exact
$n\asymp d/\eps^2$ density scale.

\citet[Theorem~1.2; Theorem~4.16 in arXiv v2]{DK24} give a polynomial-time
recursive pseudo-projection and
implicit high-order moment method for bounded spherical Gaussian mixtures, but
their density hypothesis is an improper compressed polynomial/sampling oracle
and their theorem states a generic polynomial rather than the sharp
$d/\eps^2$ sample bound.  Their work is the closest conceptual predecessor:
the present first-order-fiber/second-order-residual lemma sharpens this implicit
moment-subspace paradigm rather than introducing tensor-free moment computation
de novo.  \citet[Theorems~5 and~7]{FL23} give efficient sparse-moment recovery
without separation and a known-common-covariance Gaussian-mixture guarantee in
transportation distance.  The moment-to-parameter conversion is not the sharp
uniform Hellinger statement here; for $d>k$, their Gaussian-mixture application
invokes the dimension-reduction step of \citet{DWYZ23}.
\citet[Theorem~1.4 and Corollary~8.1]{Bakshi22} treat the more general
problem of robustly learning mixtures with arbitrary component covariances.
They use high-order Hermite-tensor flattenings to construct low-dimensional
parameter subspaces and randomized contractions in their covariance-recovery
routine.  Their separation-free proper learner pays
$d^{O(k)}\operatorname{poly}_k(1/\eps)$ rather than the sharp $d/\eps^2$
sample dependence in the spherical location setting studied here.  Thus
tensor-flattening and contraction
primitives themselves are not new.  Among polynomial-time high-dimensional
algorithms, the remaining no-separation results either return an improper
sampler or density hypothesis with non-sharp sample dependence---for example,
\citet{CKS25} use
$d^{\operatorname{poly}(k/\eps)}$ samples---while parameter-recovery guarantees
generally require separation or lower-weight assumptions.  None of these
results directly implies
\cref{thm:main}.

Relative to Bakshi et al., we avoid estimating a full ambient high-order
tensor.  Relative to Diakonikolas--Kane, we directly estimate a fixed family
of restricted one-free-index fibers as $d$-vectors and complete a subspace of
$\R^d$, rather than recursively constructing tensor-power pseudo-projections
and an improper compressed density oracle.  The fibers enlarge a coarse range
until every remaining tensor residual is quadratic in the missed component.

\renewcommand{\printhistory}{}
\end{document}